\UseRawInputEncoding
\documentclass[12pt]{article}

\usepackage{dsfont}
\usepackage{amsfonts,amsmath,amsthm}
\usepackage{algorithm}
\usepackage{algpseudocode}
\usepackage{xcolor}
\usepackage{graphicx}
\usepackage{stmaryrd}        

\usepackage[paper=a4paper,dvips,top=2cm,left=2cm,right=2cm,
foot=1cm,bottom=4cm]{geometry}

\newcommand{\vh}{{\bf h}}

\begin{document}
	\large
	
	\title{Strongly Positive Semi-Definite Tensors and Strongly SOS Tensors}
	\author{Liqun Qi\footnote{Department of Applied Mathematics, The Hong Kong Polytechnic University, Hung Hom, Kowloon, Hong Kong.
			({\tt maqilq@polyu.edu.hk}).}
		\and { \
			Chunfeng Cui\footnote{LMIB of the Ministry of Education, School of Mathematical Sciences, Beihang University, Beijing 100191 China.
				({\tt chunfengcui@buaa.edu.cn}).}
		}
	}
	\date{\today}
	\maketitle
	
	\begin{abstract}
	 We  introduce {odd-order} strongly PSD (positive semi-definite) tensors which map real vectors to nonnegative vectors.   We then introduce {odd-order} strongly SOS (sum-of-squares) tensors.   A strongly SOS tensor maps real vectors to nonnegative vectors whose components are all  SOS {polynomials.}   Strongly SOS tensors are strongly PSD tensors.   Odd order completely positive tensors are strongly SOS tensors.  We also introduce strict Hankel tensors, which are also strongly SOS tensors.   Odd order Hilbert tensors are  strict Hankel tensors.   However, the Laplacian tensor of a uniform hypergraph may not be strongly PSD.   This motivates us to study wider PSD-like tensors. A cubic tensor  is said to be a generalized PSD {}  tensor if    its  corresponding   symmetrization tensor  has no  negative   H-eigenvalue.   In the odd order case, this extension contains a peculiar tensor class, whose members have no H-eigenvalues at all.  We call such tensors barren tensors, and the other generalized PSD symmetric tensors genuinely PSD symmetric tensors.    In the odd order case, genuinely PSD tensors embrace various useful structured tensors, such as strongly PSD tensors, symmetric M-tensors and Laplacian tensors of uniform hypergraphs.
			However, there are exceptions.  It is known an even order symmetric B-tensor is a PSD tensor.  We give an example of an odd order symmetric B-tensor, which is a barren tensor.

		\medskip


		
	\end{abstract}

	\renewcommand{\Re}{\mathds{R}}
	\newcommand{\rank}{\mathrm{rank}}
	\newcommand{\X}{\mathcal{X}}
	\newcommand{\A}{\mathcal{A}}
	\newcommand{\I}{\mathcal{I}}
	\newcommand{\B}{\mathcal{B}}
	\newcommand{\C}{\mathcal{C}}
	\newcommand{\D}{\mathcal{D}}
	\newcommand{\LL}{\mathcal{L}}
	\newcommand{\OO}{\mathcal{O}}
	\newcommand{\e}{\mathbf{e}}
	\newcommand{\0}{\mathbf{0}}
	\newcommand{\dd}{\mathbf{d}}
	\newcommand{\ii}{\mathbf{i}}
	\newcommand{\jj}{\mathbf{j}}
	\newcommand{\kk}{\mathbf{k}}
	\newcommand{\va}{\mathbf{a}}
	\newcommand{\vb}{\mathbf{b}}
	\newcommand{\vc}{\mathbf{c}}
	\newcommand{\vq}{\mathbf{q}}
	\newcommand{\vg}{\mathbf{g}}
	\newcommand{\pr}{\vec{r}}
	\newcommand{\pc}{\vec{c}}
	\newcommand{\ps}{\vec{s}}
	\newcommand{\pt}{\vec{t}}
	\newcommand{\pu}{\vec{u}}
	\newcommand{\pv}{\vec{v}}
	\newcommand{\pn}{\vec{n}}
	\newcommand{\pp}{\vec{p}}
	\newcommand{\pq}{\vec{q}}
	\newcommand{\pl}{\vec{l}}
	\newcommand{\vt}{\rm{vec}}
	\newcommand{\vx}{\mathbf{x}}
	\newcommand{\vy}{\mathbf{y}}
	\newcommand{\vu}{\mathbf{u}}
	\newcommand{\vv}{\mathbf{v}}
	\newcommand{\y}{\mathbf{y}}
	\newcommand{\vz}{\mathbf{z}}
	\newcommand{\T}{\top}
	\newcommand{\R}{\mathcal{R}}

	\newtheorem{Thm}{Theorem}[section]
	\newtheorem{Def}[Thm]{Definition}
	\newtheorem{Ass}[Thm]{Assumption}
	\newtheorem{Lem}[Thm]{Lemma}
	\newtheorem{Prop}[Thm]{Proposition}
	\newtheorem{Cor}[Thm]{Corollary}
	\newtheorem{example}[Thm]{Example}
	\newtheorem{remark}[Thm]{Remark}
	
	\section{Introduction}

	Let $m, n$ be positive integers and $m, n \ge 2$.  Denote the set of all $m$th order $n$-dimensional real cubic tensors by $T_{m, n}$.
	Let $\A = \left(a_{i_1\dots i_m}\right) \in T_{m, n}$.
	If $a_{i_1\dots i_m}$ is invariant for any permutation of its indices, then $\A$ is called a symmetric tensor.   Denote the set of all $m$th order $n$-dimensional real symmetric tensors by $S_{m, n}$.   For
	$\A = \left(a_{i_1\dots i_m}\right) \in T_{m, n}$, define a homogeneous polynomial $f : {\mathbb R}^n \to {\mathbb R}$ by
	$$f(\vx) \equiv \A \vx^m = \sum_{i_1,\dots,i_m=1}^n a_{i_1\dots i_m}x_{i_1}\cdots x_{i_m}, \forall \vx \in {\mathbb R}^n.$$
	It is also known that there is a unique symmetric tensor $\B \in S_{m, n}$ such that
	$$f(\vx) \equiv \A \vx^m \equiv \B \vx^m.$$
	The symmetric tensor $\B$ is called the symmetrization of $\A$, and denoted as $\B = {\rm Sym}(\A)$.
	
	If {$\A\vx ^m\ge 0$} for all $\vx \in {\mathbb R}^n$, then $\A$ is called positive semi-definite.
	If {$\A\vx^m > 0$} for all $\vx \in {\mathbb R}^n, \vx \not = \0$, then $\A$ is called positive definite.
	Then this is only meaningful for even-order tensors.   Positive  semi-define tensors and positive definite tensors were introduced for symmetric tensors in \cite{Qi05}, and generalized to non-symmetric tensors in \cite{HHLQ13}.    They have applications in automatic control, spectral hypergraph theory, magnetic resonance imaging and physics.  See Page 10 of \cite{QL17}.
	
	In this paper, we consider positive semi-definiteness of odd-order tensors, by introducing the following tensor class.
		
		\begin{Def}
			Suppose that $m$ is odd.   Let $\A \in  {S_{m, n}}$.   Define $F: {\mathbb R}^n \to {\mathbb R}^n$ as $F(\vx) = \A\vx^{m-1}$ for $\vx \in {\mathbb R}^n$.   We say that $\A$ is {\bf strongly positive semi-definite} if $F(\vx) \ge \0$ for any $\vx \in {\mathbb R}^n$, and $\A$ is {\bf strongly positive definite} if $F(\vx) > \0$ for any
			$\vx \in {\mathbb R}^n$ and $\vx \not = \0$.
		\end{Def}
		In a certain sense, this definition has more geometrical meanings, comparing with generalized positive semi-definiteness and genuinely positive semi-definiteness.
		
		We show that completely positive semi-definite tensors are strongly positive semi-definite tensors, a strongly positive semi-definite tensor may have negative entries,  and the Laplacian tensor of an odd order uniform hypergraph may not be a strongly positive semi-definite tensor.   The set of strongly positive semi-definite tensors of the same order and dimension form a closed convex cone.   We also study its dual cone.
		
		A property stronger than positive semi-definiteness is the SOS (sum-of-squares) property, which is somehow associated with David Hilbert \cite{Hi88}.  Hence, for odd order cubic tensors, we further introduce strongly SOS tensors.  Strongly SOS tensors are strongly positive semi-definite tensors.   We show that completely positive tensors are strongly SOS tensors.  We introduce strict Hankel tensors which are also strongly SOS tensors.  We show that the Hilbert tensors are strict Hankel tensors.
		
		As some important structured tensors, such as the Laplacian tensor of an odd order uniform hypergraph, may not be a strongly positive semi-definite tensors, we thus consider some wider positive semi-definiteness-like properties for odd order tensors.
	
	Let $\A = \left(a_{i_1\dots i_m}\right) \in T_{m, n}$.   For $\vx \in {\mathbb C}^n$, define
	$$\A\vx^{m-1} = \left( \sum_{i_2,\dots, i_m=1}^n a_{ii_2\dots i_m}x_{i_2}\cdots x_{i_m}\right),$$
	and $\vx^{[m-1]} = \left( x_i^{m-1} \right)$.  If there is {a nonzero vector $\vx \in {\mathbb C}^n$} and $\lambda \in {\mathbb C}$ such that
	\begin{equation} \label{eig}
		\A \vx^{m-1} = \lambda \vx^{[m-1]},
	\end{equation}
	then $\lambda$ is called an eigenvalue of $\A$, and $\vx$ is called an eigenvector of $\A$, associated with the eigenvalue $\vx$.   If furthermore $\vx$ is a real vector, then $\lambda$ is called an H-eigenvalue of $\A$ and $\vx$ is called an H-eigenvector of $\A$, associated with the H-eigenvalue $\lambda$.  Then an H-eigenvalue is a real number.   The largest modulus of the eigenvalues of $\A$ is called the spectral radius of $\A$, and denoted as $\rho(\A)$.   Suppose that $\A \in S_{m, n}$ and $m$ is even.  Then $\A$ always has H-eigenvalues, $\A$ is positive semi-definite if and only if all of its H-eigenvalues are nonnegative, and $\A$ is positive definite if and only if all of its H-eigenvalues are positive.
	
	The study of spectral hypergraph theory motivated to extend the concepts of positive semi-definiteness and positive definiteness to odd order tensors.  It is known that Laplacian and signless Laplacian tensors of even order uniform hypergraphs are positive semi-definite.   To extend the concepts of positive semi-definiteness and positive definiteness to Laplacian and signless Laplacian tensors of odd order uniform hypergraphs, in \cite{QL17}, {\bf generalized positive semi-definite and definite tensors} were introduced.
	Let $\A \in T_{m, n}$.   We say $\A$ is a generalized positive semi-definite tensor if Sym$(\A)$ has no negative H-eigenvalues, and say $\A$ is a generalized positive definite tensor if Sym$(\A)$ has no non-positive H-eigenvalues.   It was {presented} in \cite{QL17} that many structured tensors, such as Laplacian and signless Laplacian tensors of uniform hypergraphs, strong Hankel tensors, complete Hankel tensors, and completely positive tensors are all generalized positive semi-definite tensors.
	
	However, the definition of generalized positive semi-definite tensors may involve some non-genuine candidates.  For example, on Page 17 {of \cite{QL17}}, an example of an odd order copositive tensor without H-eigenvalues was described.  In the even order case, a copositive tensor is not a positive semi-definite tensor in general.   However, by this definition, an odd order copositive tensor is a generalized positive semi-definite tensor.   Hence, we should somehow exclude such non-genuine candidates.
	
	\begin{Def}
		Let $\A \in T_{m, n}$.  We say that $\A$ is a {\bf genuinely positive semi-definite tensor} if Sym$(\A)$ has at least one H-eigenvalue and all of its H-eigenvalues are nonnegative.   We say that $\A$ is a {\bf genuinely positive definite tensor} if Sym$(\A)$ has at least one H-eigenvalue and all of its H-eigenvalues are positive.
	\end{Def}

	In the next section, we  introduce barren tensors, which are tensors without H-eigenvalues.  Then
	the class of generalized positive semi-definite tensors is partitioned into two subclasses: the class of
	genuinely positive semi-definite tensors, and the class of barren tensors.  The behaviours of barren tensors are somewhat eccentric.   Also, they are difficult for spectral analysis.   Thus, it is necessary to distinguish barren tensors from genuinely positive semi-definite tensors.   We enumerate a number of structured tensor classes, which are genuinely positive semi-definite tensors.    We also present an example
	of an odd symmetric B-tensor, which is a barren tensor, though an even order symmetric B-tensor is always positive definite.
	
	In Section 3, we show that strongly positive semi-definite tensors are genuinely positive semi-definite tensors, and present their properties.
	
	We introduce strongly SOS tensors and study their properties in Section 4.
	
	Some final remarks are made in Section 5.

	
	\section{Genuinely Positive Semi-Definite Tensors and Barren Tensors}
	
	We may call a cubic tensor $\A \in T_ {m, n}$ a {\bf barren tensor} if Sym$(\A)$ has no H-eigenvalue at all.  The odd order copositive tensor on Page 17 of \cite{QL17} is such a tensor.  If $\A$ is a barren tensor, then $m$ must be odd.   A generalized positive semi-definite tensor is a genuinely positive semi-definite tensor if it is not a barren tensor.  Furthermore, by definition, a barren tensor is in fact a generalized positive definite tensor.   At this moment, barren tensors are essentially unexplored except the example on Page 17 of \cite{QL17}.
	
	The following proposition {shows} that the behaviours of genuinely positive semi-definite tensors and barren tensors are very different.  In fact, the behaviours of barren tensors are somewhat eccentric.
	
	\begin{Prop}
		If $\A \in T_{m, n}$ is a nonzero genuinely positive definite tensor, then $-\A$ is not a  genuinely positive definite tensor.   On the other hand, if $\A \in T_{m, n}$ is a barren tensor, then $-\A$ is also a barren tensor.
	\end{Prop}
	\begin{proof}
		The conclusions follow directly from definitions.
	\end{proof}
	
	Also, as barren tensors have no H-eigenvalues, it is hard to make spectral analysis for them.
	{According to {\cite{CPT08, QL17}}, a nonnegative tensor possesses a nonnegative H-eigenvalue. Nevertheless, distinguishing barren tensors and genuinely positive semi-definite tensors is still a challenging task.}
	
	%
	
	On the other hand, we may enumerate important structured tensors which are genuinely positive semi-definite.

	\begin{example}
		Symmetric M-tensors.  A cubic tensor $\A \in T_{m, n}$ is called an M-tensor if there exists a nonnegative tensor $\B \in T_{m, n}$ and a positive scalar $s \ge \rho(\B)$ such that $\A = s\I - \B$, where $\I$ is the identity tensor in $T_{m, n}$.  If $s > \rho(\B)$, then $A$ is called a strong M-tensor.  See \cite{QL17}. As a nonnegative tensor always has H-eigenvalues (See Theorem 2.4 of \cite{QL17}), a
		symmetric M-tensor is a genuinely positive semi-definite tensor, and a symmetric strong M-tensor is a genuinely positive definite tensor.  Note that M-tensors have wide applications \cite{DQW13, DW16, ZQZ14}.
	\end{example}
	
	\begin{example}
		Laplacian tensors of uniform hypergraphs.   A uniform hypergraph $G = (V, E)$ has a vertex set $V = \{ 1, 2, \dots, n \}$ and an edge set $E = \{ e_1, \dots, e_p \}$, where each edge $e_i$ is a subset of $V$, with a cardinality $m$.  Such a uniform hypergraph is called an $m$-graph.  The adjacency tensor $\A$ of $G$, $\A = \left(a_{i_1\dots i_m}\right) \in S_{m, n}$ is defined by
		$a_{i_1\dots i_m} = {1 \over (m-1)!}$, if $(i_1, \dots, i_m) \in E$, and $a_{i_1\dots i_m} = 0$ otherwise.
		For a vertex $i \in V$, its degree $d(i)$ is defined by $d(i)= \left| \{ e_k : i \in e_k \in E \} \right|$.
		The degree tensor $\D$ of $G$, is an $m$th order $n$-dimensional diagonal tensor with its $i$th diagonal entry as $d(i)$.   Then the Laplacian tensor {$\LL$} of $G$ is defined as ${\LL} = \D - \A$.  Since a Laplacian tensor is a symmetric M-tensor, it is a genuinely positive semi-definite tensor.  As spectral analysis of Laplacian tensors plays a central role in spectral hypergraph theory \cite{CvB24, CD12, LSF24, LCC17, Ni14, YSS16}, it is important to identify Laplacian tensors are genuinely positive semi-definite tensors.
	\end{example}
	
	\begin{example}
		Nonnegative diagonal tensors.
		{All eigenvalues of a diagonal tensor are the diagonal entries of that diagonal tensor,  with unit vectors as their eigenvectors \cite{QL17}.}
		Hence, a nonnegative diagonal tensor is a genuinely positive semi-definite tensor, and a positive diagonal tensor is a genuinely positive definite tensor.  Note the degree tensor of a uniform hypergraph is such a tensor.
	\end{example}
	
	\begin{example}
		Nonnegative generalized positive semi-definite tensors. As a nonnegative tensor always has an H-eigenvalue \cite{QL17}, a nonnegative generalized positive semi-definite tensor is a genuinely positive semi-definite tensor, and a positive generalized positive semi-definite tensor is a genuinely positive semi-definite tensor.
	\end{example}
	
	\begin{example}
		Completely positive tensors.  Let $\vu = (u_1, \dots, u_n)^\top \in {\mathbb R}^n$.   Then we use $\vu^m$ to denote an $m$th order tensor
		$\vu^m = \left(u_{i_1}u_{i_2}\cdots u_{i_m}\right)$.    Suppose that {for a given positive integer  {$r$},} we have nonnegative vectors $\vu^{(1)}, \dots, {\vu^{(r)}}$ such that
		\begin{equation} \label{complete}
			\A = {\sum_{l=1}^r} \left(\vu^{(l)}\right)^m.
		\end{equation}		
		Then $\A$ is an $m$th order completely positive tensor.    Furthermore, if {$r\ge n$ and} $\left\{ \vu^{(1)}, \dots, {\vu^{(r)}} \right\}$ spans ${\mathbb R}^n$, then $\A$ is an $m$th order strongly completely positive tensor.   By definition, completely positive tensors are nonnegative tensors.  By Theorem 6.1 of \cite{QL17}, all the H-eigenvalues of a completely positive tensors are nonnegative.   By Theorem 6.2 of \cite{QL17}, all the H-eigenvalues of a strongly completely positive tensors are positive.
		Thus, completely positive tensors are genuinely positive semi-definite tensors, and strongly completely positive tensors are genuinely positive definite tensors.
	\end{example}
	
	\begin{example}
		Positive Cauchy tensors.  An $m$th order positive Cauchy tensor $\A = \left(a_{i_1\cdots i_m}\right)$ {can be described by}
		\begin{equation} \label{Cauchy}
			a_{i_1\dots i_m} = {1 \over c_{i_1} + \dots + c_{i_m}},
		\end{equation}
		where $\vc = (c_1, \dots, c_n)^\top$ is a positive vector in ${\mathbb R}^n_{++}$, called the generating vector of $\A$.  By Theorem 6.14 of \cite{QL17}, a positive Cauchy tensor is a completely positive tensor, thus a genuinely positive semi-definite tensor.  Furthermore, if all the components of $\vc$ are distinct, then $\A$ is a strongly completely positive tensor, hence a genuinely positive definite tensor.
	\end{example}
	
	\begin{example}
		The Hilbert tensor.   The Hilbert tensor can be regarded as a special positive Cauchy tensor with $c_i = i - {m-1 \over m}$ \cite{QL17}.  Since all the components of $\vc$ are distinct, the Hilbert tensor is a strongly completely positive tensor, hence a genuinely positive definite tensor.
	\end{example}
	
	\begin{example}
		The Pascal tensor.  The $m$th order $n$-dimensional Pascal tensor {$\A = \left(a_{i_1\cdots i_m}\right)$} can be described by
		$${a_{i_1\cdots i_m}} = {(i_1+\dots + i_m -m)! \over (i_1-1)!\dots (i_m-1)!}.$$
		By {Proposition 6.19 of} \cite{QL17}, the Pascal tensor is a completely positive tensor, hence a genuinely positive semi-definite tensor.
	\end{example}
	
	\begin{example}
		The Lehmer tensor.  The $m$th order $n$-dimensional Lehmer tensor $\A = \left(a_{i_1\cdots i_m}\right)$ can be described by
		$$a_{i_1\cdots i_m} = {\min \{i_1, \dots, i_m \} \over \max \{i_1, \dots, i_m \}}.$$
		By {Proposition 6.20 of}  \cite{QL17}, the Lehmer tensor is a completely positive tensor, hence a genuinely positive semi-definite tensor.
	\end{example}
	
	With the above examples, we may think that in general for a symmetric tensor family, if their even order members are positive semi-definite, then their odd order members would be genuinely positive semi-definite.  Then this conclusion is wrong.   We have the following example to show this.
	
	\begin{example}
		An odd order symmetric B-tensor which is a barren tensor.    It was proved in \cite{QS14} that an even order symmetric B-tensor is positive definite.   Also see Section 5.5 of \cite{QL17}.  Let $\B = \left(b_{ijk}\right) \in S_{3, 2}$ be defined by
		$$b_{111} = 10,\ b_{222} = 6,\ b_{112} = b_{121} = b_{211} = - \sqrt{3},\ {\rm and}\ b_{221} = b_{212} = b_{122} = \sqrt{3}.$$
		Then we have
		$$b_{111} + b_{112} + b_{121} + b_{122} = 16-2\sqrt{3} > 0,$$
		$$b_{211} + b_{212} + b_{221} + b_{222} = 6 + \sqrt{3} > 0,$$
		$${1 \over 4}\left(16-2\sqrt{3}\right) > b_{112}, b_{121}, b_{122},$$
		and
		$${1 \over 4}\left(6 + \sqrt{3}\right) > b_{211}, b_{212}, b_{221}.$$
		By {the definition in Section 5.5 of} \cite{QL17}, $\B$ is a symmetric B-tensor.  Assume that $\B$ has an H-eigenvalue $\lambda \in {\mathbb R}$ with an H-eigenvector $\vx \in {\mathbb R}^2$.  Then we have
		$$\left\{ \begin{aligned} 10x_1^2 -2\sqrt{3}x_1x_2 + \sqrt{3}x_2^2 & = & \lambda x_1^2, \\ -\sqrt{3}x_1^2 + 2\sqrt{3}x_1x_2  + 6x_2^2 & = & \lambda x_2^2. \end{aligned} \right.$$
		If $x_2 = 0$, then $x_1 = 0$.  This is impossible.  Thus, we may assume $x_2 \not = 0$.  Let $z = {x_1 \over x_2}$.  Dividing $x_2^2$ on the both sides of the above equations, we have the following equations:
		$$\left\{ \begin{aligned} (10-\lambda)z^2 -2\sqrt{3}z + \sqrt{3} & = & 0, \\ -\sqrt{3}z^2 + 2\sqrt{3}z  + (6-\lambda) & = & 0. \end{aligned} \right.$$
		Since $\lambda$ and $z$ are real, by checking the discriminants of these two quadratic equations, we have
		$$12 - 4\sqrt{3}(10-\lambda) \ge 0,\ 12+ 4\sqrt{3}(6-\lambda) \ge 0. $$
		This leads to a contradiction $6 + \sqrt{3} \ge \lambda \ge 10 - \sqrt{3}$.
		Hence, $\B$ has no H-eigenvalues, and is a barren tensor.
	\end{example}
	
	Is an odd order symmetric B-tensor always a generalized positive definite tensor?   This remains an open question.

	\section{Strongly Positive Semi-Definite Tensors}
	
	We study the properties of strongly positive semi-definite tensors and strongly positive definite tensors in this section.  By definition{, these} tensors are odd order symmetric tensors.
	
	We have the following propositions.
	
	\begin{Prop} \label{cp1}
		An odd order symmetric tensor $\A = \left(a_{i_1\dots i_m}\right)$ is strongly positive semi-definite if and only if $\A$ reduces to an even order positive semi-definite symmetric tensor whenever one of the indices of its entries is fixed.
		An odd order symmetric tensor $\A = \left(a_{i_1\dots i_m}\right)$ is strongly positive definite if and only if $\A$ reduces to an even order positive definite symmetric tensor whenever one of the indices of its entries is fixed.
	\end{Prop}
	\begin{proof}
		Define $\vy = F(\vx) \equiv \A\vx^{m-1}$.  The element $y_i$ is nonnegative (resp. positive) if and only if the $(m-1)$th order tensor  derived by fixing $i_1=i$ is positive semi-definite (resp. positive definite).
		The above results follows directly from the fact that $\A$ is symmetric and by fixing any of the indices, the resulted $(m-1)$th order tensor is  a  positive semi-definite (resp. positive definite) symmetric tensor if {and only if} $\A$ is strongly positive semi-definite (resp. strongly positive definite).
	\end{proof}
	
	\begin{Prop}\label{cp2}
		Suppose that   $\A = \left(a_{i_1\dots i_m}\right)$ is a {symmetric and} strongly positive semi-definite tensor.   Then $a_{i_1\dots i_m}$ is a nonnegative number if $m-1$ of $i_1, \dots, i_m$ are the same.   If furthermore $\A$ is strongly positive definite, then $a_{i_1\dots i_m}$ is a positive number if $m-1$ of $i_1, \dots, i_m$ are the same.
	\end{Prop}
	\begin{proof}
		The above results follow directly from Proposition~\ref{cp1}.
	\end{proof}
	
	Furthermore, we have the following theorem.
	
	\begin{Thm}
		Suppose that $\A$ is a strongly positive semi-definite tensor.   Then $\A$ is a genuinely positive semi-definite tensor.   If furthermore $\A$ is strongly positive definite, then $\A$ is a genuinely positive definite tensor.
	\end{Thm}
	\begin{proof}
		By Theorem 3.10 of \cite{Zh11}, a strongly positive semi-definite tensor has at least one H-eigenvalue.
		Suppose that $\lambda$ is an H-eigenvalue of $\A$ with an H-eigenvector $\vx$.  Then for $i \in [n]$, we have
		$$\sum_{i_2, \dots, i_m =1}^n a_{ii_2\dots i_m}x_{i_2}\cdots x_{i_m} = \lambda x_i^{m-1}.$$
		By definition, if $\A$ is strongly positive semi-definite, then $\lambda \ge 0$, and if $\A$ is strongly positive definite, then $\lambda > 0$.  By the definition of genuinely positive semi-definite and definite tensors, we have the conclusion.
	\end{proof}

	

	We now consider completely positive tensors.

	\begin{Thm} \label{cp}
		A completely positive tensor $\A = \left(a_{i_1\dots i_m}\right) \in S_{m,n}$ is a strongly positive semi-definite tensor.
		
		Furthermore, suppose that $\A = \left(a_{i_1\dots i_m}\right) \in S_{m,n}$ is a completely positive tensor with the form (\ref{complete}). Let $\Gamma_i:=\{l\in [r]: u_i^{(l)}\neq 0\}$ for each $i\in [n]$. If rank$\left([\vu^{(l)}]_{l\in \Gamma_i}\right)=n$ for all $i\in [n]$, then $\A$ is a strongly positive definite tensor.
	\end{Thm}
	\begin{proof}   Suppose that $\A = \left(a_{i_1\dots i_m}\right)$ is a completely positive tensor with the form (\ref{complete}).
		%
		For any $i\in [n]$, the $i$th slice matrix of $\A$ in any mode (by the symmetry of the tensor) takes the form of
		$$A(i) = \sum_{l=1}^r  u_i^{(l)}{\left(\vu^{(l)}\right)^{(m-1)}.}$$
		The positive semi-definiteness of each $A(i)$ follows directly by the nonnegativity of all $\vu^{(l)}$'s. Thus, $\A$ is a strongly positive semi-definite tensor by definition.
		
		Furthermore, one has
		$$A(i) = \sum_{l=1}^r  u_i^{(l)} {\left(\vu^{(l)}\right)^{(m-1)}} = \sum_{l\in \Gamma_i} u_i^{(l)}{\left(\vu^{(l)}\right)^{(m-1)}}.$$ Note that $u_i^{(l)}>0$ for each $l\in \Gamma_i$, and rank$\left([\vu^{(l)}]_{l\in \Gamma_i}\right)=n$. Thus $A(i)$ is positive definite for all $i\in [n]$. The desired result follows readily from the definition of strongly positive definite tensors.
	\end{proof}
	
	\begin{Cor}\label{cp3} If $\A$ is a strongly completely positive tensor with the form (\ref{complete}), and $\vu^{(1)}, \dots, \vu^{(r)}$ are all positive {vectors}, then $\A$ is a strongly positive definite tensor.
	\end{Cor}		
	\begin{proof}
		Obviously, for a strongly completely positive tensor $\A$ with all $\vu^{(l)}$'s positive, the index set $\Gamma_i$ in Proposition \ref{cp3} turns out to be $[r]$ for each $i$, and rank$\left([\vu^{(l)}]_{l\in [r]}\right)=n$ holds since all $\vu^{(l)}$'s  span the entire space $\mathbb R^n$ by the definition of strongly completely positive tensors.
	\end{proof}
	
	\begin{Cor}
		Suppose that $\A$ is a positive Cauchy tensor described by (\ref{Cauchy}), and all the components of $\vc$ are distinct.   Then $\A$ is a strongly positive definite tensor.
	\end{Cor}
	
	\begin{Cor}
		The Hilbert tensor is a strongly positive definite tensor.
	\end{Cor}

	We draw the  relations among   generalized positive semi-definite tensors, genuinely positive semi-definite tensors, barren tensors, strongly positive semi-definite tensors, and completely positive tensors for both even and odd order cases in Fig.~\ref{fig:generalized}.
		
		\begin{figure}
			\begin{center}
				\includegraphics[width=\linewidth]{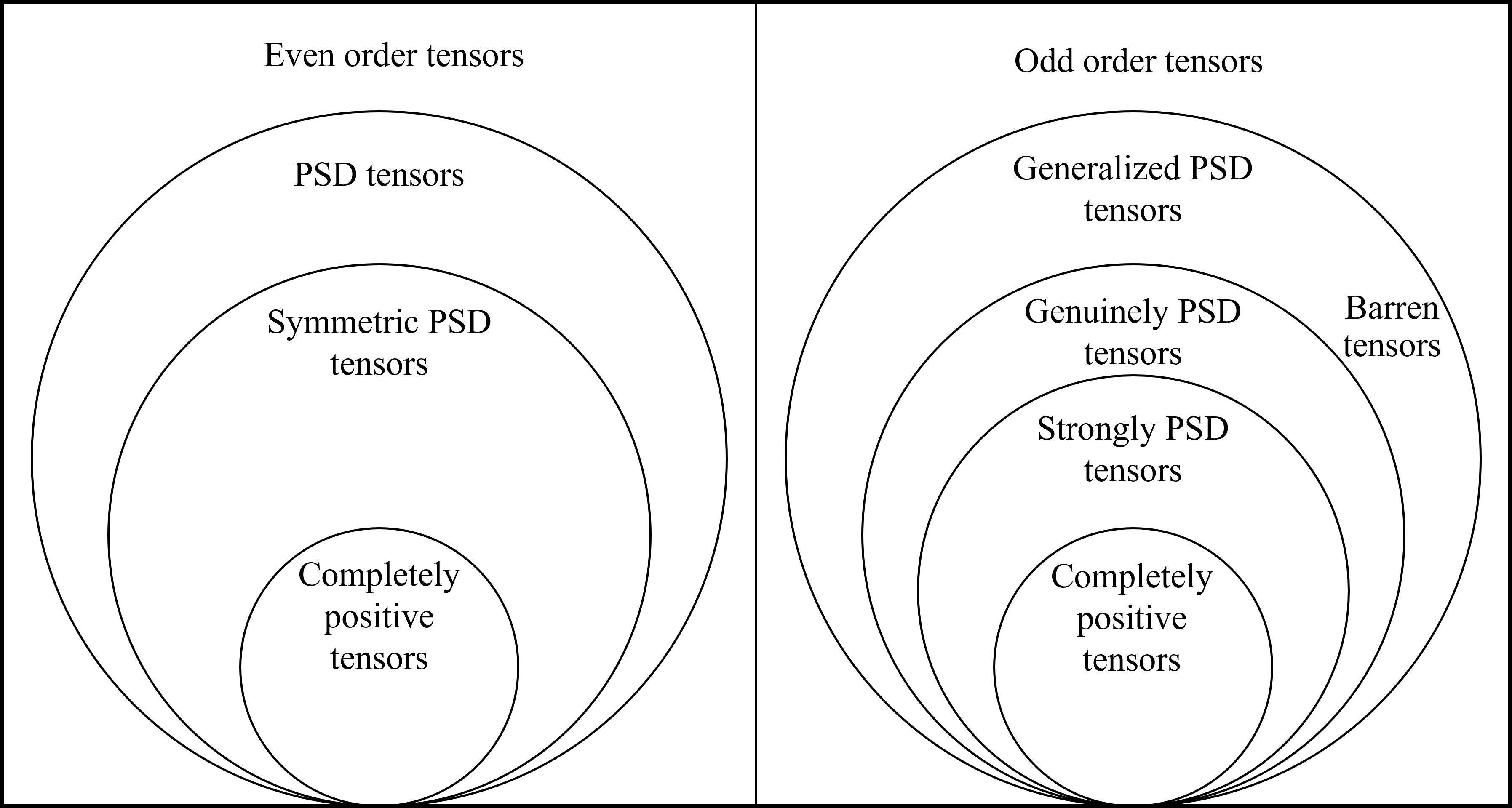}
			\end{center}
			\caption{The relationship among generalized positive semi-definite tensors, genuinely positive semi-definite tensors, barren tensors, strongly positive semi-definite tensors, and completely positive tensors for both even and odd order {cases. Here, PSD is the abbreviation for positive semi-definite.}}
			\label{fig:generalized}
		\end{figure}

	So far, all the examples of strongly positive semi-definite tensors are nonnegative tensors.  The following example shows that there are strongly positive semi-definite tensors which are not nonnegative tensors.   We also show that this example is genuinely positive semi-definite, i.e., it is not a barren tensor.

	\begin{example}
		Let the $\A=(a_{i_1i_2i_3})$ be defined as follows,
		\begin{equation}
			a_{i_1i_2i_3} = \left\{
			\begin{array}{ll}
				c_1, & \text{if } i_1=i_2=i_3, \\
				c_2,& \text{{else if }} i_1=i_2 \text{ or }  i_2=i_3  \text{ or }  i_1=i_3, \\
				-1, & \text{otherwise},
			\end{array}
			\right.
		\end{equation}
		where $c_1,c_2\ge 0$, {and $c_1-3c_2-3\ge 0$}.
		Then we have
		\begin{eqnarray*}
			\sum_{i_2,i_3=1}^n a_{i_1i_2i_3} x_2x_3 &=&  (2c_2+2){x_{i_1}} \left(\sum_{i=1}^n x_i \right) + (c_2+1)\left(\sum_{i=1}^n x_i^2\right)  \\
			&&-    \left(\sum_{i=1}^n x_i \right)^2  +(c_1-3c_2-3) {x_{i_1}^2}.
		\end{eqnarray*}
		{Without loss of generality,} assume that $\sum_{i=1}^n x_i^2 =1$ and $\sum_{i=1}^n x_i = \alpha$.
		Then we have $-\sqrt{n}\le \alpha \le \sqrt{n}$, $-1\le x_i\le 1$, and
		\begin{eqnarray*}
			\sum_{i_2,i_3=1}^n a_{i_1i_2i_3} x_2x_3 &=&  (2c_2+2){x_{i_1}}\alpha +(c_1-3c_2-3) {x_{i_1}^2}-\alpha^2 +c_2+1\\
			&\ge&  -\frac{(c_2+1)^2}{c_1-3c_2-3}\alpha^2 -\alpha^2 +c_2+1\\
			&\ge& c_2+1-n-\frac{(c_2+1)^2}{c_1-3c_2-3} n.
		\end{eqnarray*}
		Therefore, when $c_2+1\ge 2n$, $c_1\ge 3c_2+3+(c_2+1)^2\ge 6n+4n^2$, $\sum_{i_2,i_3=1}^n a_{i_1i_2i_3} x_2x_3$ will always be nonnegative. Thus, $\A$ is  strongly positive semi-definite.
		
	\end{example}
	
	Denote the set of $m$th order $n$-dimensional strongly positive semi-definite tensors by SPSD$_{m, n}$, and the set of $m$th order $n$-dimensional strongly positive definite tensors by SPD$_{m, n}$.  Then we have the following theorem.
	
	\begin{Thm}
		The set SPSD$_{m, n}$ is a closed convex cone.  Suppose that $\A, \B \in$ SPSD$_{3, n}$. Then the Hadamard product $\C$ of $\A$ and $\B$ is still in SPSD$_{3, n}$.   If $\A, \B \in$ SPD$_{3, n}$, then $\C  \in$ SPD$_{3, n}$.
	\end{Thm}
	\begin{proof}  The first conclusion follows {from the} definition.  The other two conclusions follow from Proposition \ref{cp1} and the Schur product theorem in matrix analysis \cite{HJ13}.
	\end{proof}
	
	{When $m\ge 5$, the Hadamard product $\C$ of $\A,\B\in$SPSD$_{m, n}$  may not be in SPSD$_{m, n}$. See Example 5.7 in \cite{QL17} for an example that the Hadamard product of two symmetric positive semi-definite tensors of order $m=4$  fail to be positive semi-definite.}
	
	{Furthermore,} we may study the dual cone of SPSD$_{m, n}$.
	Before proceeding, let us recall the dual  {cone} of the positive semi-definite tensors, referred to as PSD$_{m, n}$ for an even number $m$. By Proposition 5.3 in  \cite{QL17}, PSD$_{m, n}$ is   a closed convex pointed cone, and its dual cone,
	termed $V_{m,n}$, has the form
	\begin{equation}
		V_{m,n}=\left\{\sum_{i=1}^l \left(x^{(i)}\right)^m: x^{(i)}\in\mathbb R^n, i=1,\dots,l\right\},
	\end{equation}
	where $l$ is the dimension of the space $S_{m,n}$. Furthermore,  {$V_{m,n}\subseteq$PSD$_{m,n}$ and the equality holds if and only if $m=2$.}
	Let $m$ be an odd number. Define
	\begin{equation}
		SV_{m,n}=\{\A=(a_{i_1,\dots,i_m}):  \A_i=(a_{ii_2,\dots,i_m})  \in V_{m-1,n} \text{ for all } i=1,\dots,n\}.
	\end{equation}
	Then we have the following results.
	\begin{Prop}\label{prop:dual_cone}
		For any odd number $m\ge 3$, we have $SV_{m,n}$ is in the dual cone of SPSD$_{m, n}$. Furthermore,
		{we have $SV_{m,n}=$SPSD$_{m,n}$ and $SV_{3,n}=$SPSD$_{3,n}$.}
	\end{Prop}
	\begin{proof}
		The   conclusions follow from the definition.
	\end{proof}

	The following example shows that the Laplacian tensor of an odd order uniform hypergraph may be not a strongly positive semi-definite tensor.   As the Laplacian tensor of a uniform hypergraph is a symmetric M-tensor, an M-tensor may be not a strongly positive semi-definite tensor too.
	
	\begin{example}
		Consider a 3-uniform hypergraph $G = (V, E)$, where $V = \{ 1, 2, 3 \}$ and $E = \{ (1, 2, 3) \}$.  Let its Laplacian tensor be $\LL = \left(l_{ijk}\right) \in S_{3,3}$.   Then we have
		$l_{111} = l_{222} = l_{333} = 1$, $l_{123} = l_{213} = l_{132} = l_{231} = l_{312} = l_{321} = -{1 \over 2}$ and $l_{ijk} = 0$ otherwise.   Fix $i = 1$.  Then $\LL$ reduces to a matrix
		$$\begin{bmatrix} 1 & 0 & 0 \\ 0 & 0 & - {1 \over 2} \\ 0 & - {1 \over 2} & 0 \end{bmatrix},$$
		which is not positive semi-definite.  By Proposition \ref{cp1}, $\LL$ is not strongly positive semi-definite.
	\end{example}
	
	\section{Strongly SOS Tensors and Strict Hankel Tensors}
	
	We now give the following definition.
	
	\begin{Def}
		Suppose that $m$ is odd.   Let $\A \in  {S_{m, n}}$.   Define $F: {\mathbb R}^n \to {\mathbb R}^n$ as $F(\vx) = \A\vx^{m-1}$ for $\vx \in {\mathbb R}^n$.   We say that $\A$ is a {\bf strongly SOS (sum-of-squares) tensor} if $F_i(\vx)$ is an SOS polynomial for $i = 1, \dots, n$.
	\end{Def}
	
	By definition, a strongly SOS tensor is a strongly positive semi-definite tensor.   Furthermore, we have the following proposition.
	
	\begin{Prop}
		Suppose $\A\in S_{m,n}$ is a  strongly positive semi-definite tensor. If either of the following three cases holds:
		\begin{itemize}
			\item[(i)] $n=2$,
			\item[(ii)] $m=3$,
			\item[(iii)] $m=5$ and $n=3$,
		\end{itemize}
		then $\A$ is also a strongly SOS tensor.
	\end{Prop}
	\begin{proof}
		As pointed out by David Hilbert \cite{Hi88}, $F_i(\vx)$ is an SOS polynomial  in all the above three cases.
	\end{proof}

\begin{Cor}
Example 3.8 is also a strongly SOS tensor.  Thus a strongly SOS tensor may contain negative entries.
\end{Cor}
	
	{
		We now consider completely positive tensors.

		\begin{Thm} \label{cp:SOS}
			A completely positive tensor $\A = \left(a_{i_1\dots i_m}\right) \in S_{m,n}$ is a strongly SOS tensor.
		\end{Thm}
		\begin{proof}   Suppose that $\A = \left(a_{i_1\dots i_m}\right)$ is a completely positive tensor with the form (\ref{complete}).
			%
			For any $i\in [n]$, the $i$th slice matrix of $\A$ in any mode (by the symmetry of the tensor) takes the form of
			$$A(i) = \sum_{l=1}^r  u_i^{(l)}{\left(\vu^{(l)}\right)^{(m-1)}.}$$
			Therefore, $F_i(\vx) =  \sum_{l=1}^r  u_i^{(l)}\left(\left(\left(\vu^{(l)}\right)^{\top}\vx\right)^{\frac{m-1}2} \right)^2$.
			This completes the proof.
		\end{proof}
	}
	
	Hankel tensors are important in applications \cite{PDV05}.  In particular, even order strong Hankel tensors and SOS tensors \cite{QL17}.   We now extend this to the odd order case.
	
	A Hankel tensor $\A = \left(a_{i_1\dots i_m}\right) \in T_{m, n}$ is a symmetric tensor, with a generating vector $\vh = \left(h_0, \dots, h_{(n-1)m} \right)^\top\in {\mathbb R}^{(n-1)m + 1}$, such that
	$$a_{i_1\dots i_m} = h_{i_1+\dots+ i_m},$$
	for $i_1, \dots, i_m = 0, \dots, n-1$.  Then $\vh$ is also a generating vector of a Hankel matrix $A = (a_{ij}) \in T_{2, \lceil{(n-1)m \over 2}\rceil+1}$, defined by
	$$a_{ij} = h_{i+j},$$
	for $i,j = 0,\dots, \lceil{(n-1)m \over 2}\rceil$. Here, if $m$ is odd and $n$ is even, $h_{(n-1)m+1}$ is an arbitrary real number.   We say that $A$ is the associated Hankel matrix of the Hankel tensor $\A$.   If $A$ is positive semi-definite, then we say that $\A$ is a strong Hankel tensor.   An even order strong Hankel tensor is an SOS tensor.   See Theorem 5.55 of \cite{QL17}.
	
	Now, let $\A = \left(a_{i_1\dots i_m}\right)\in S_{m, n}$ be an odd order Hankel tensor, i.e., $m$ is odd.   Denote $\A_i = \left(a_{ii_2\dots i_m}\right)\in S_{m-1, n}$ be the $(m-1)$th order Hankel tensor, with the first index $i_1$ being fixed at $i_1 = i$.  If $\A_i$ for $i = 0, \dots, m-1$, are all strong Hankel tensor, then $\A$ is called a {\bf strict Hankel tensor}. Clearly, strict Hankel tensors are strongly SOS tensors.    We have the following theorem to display some strict Hankel tensors.
	
	\begin{Thm}
		An odd order Hilbert tensor $\A \in S_{m, n}$ is a strict Hankel tensor, thus a strongly SOS tensor.
	\end{Thm}
	\begin{proof}
		Consider the associated Hankel matrix $A_i$ for $\A_i$, with $i = 0, \dots, n-1$.  Then any principal minor of $A_i$ is a principal minor of a big Hilbert matrix.   Such a principal minor must be positive as the Hilbert matrix is positive definite \cite{QL17}.   This implies that $A_i$ is positive definite.  Then $\A_i$ is a strong Hankel tensor for $i = 0, \dots, n-1$.   Hence, $\A$ is a strict Hankel tensor, thus a strongly SOS tensor.
	\end{proof}
	
	{We draw the  relations among   strongly positive semi-definite tensors, strongly sum of squares tensors, completely positive tensors, and strict Hankel tensors  for odd order tensors in Fig.~\ref{fig:S_SOS}.
	
	\begin{figure}
		\begin{center}
			\includegraphics[width=0.5\linewidth]{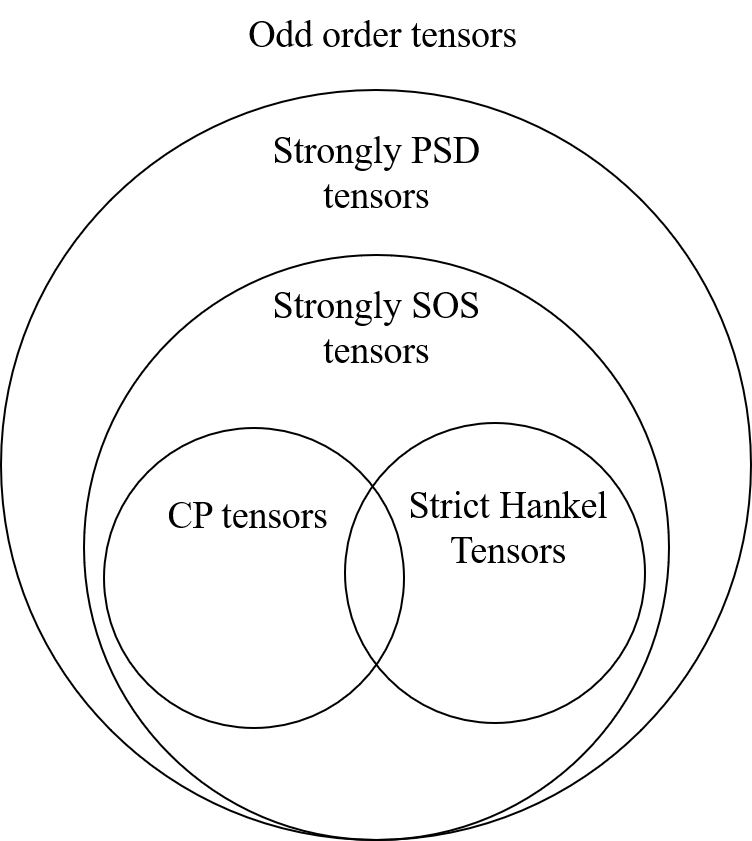}
		\end{center}
		\caption{The  relations among   strongly positive semi-definite tensors, strongly sum of squares tensors, completely positive tensors, and strict Hankel tensors  for odd order tensors.}
		\label{fig:S_SOS}
	\end{figure}
}

	\section{Final Remarks}
	
	In this paper,  we introduced strongly positive semi-definite tensors and strongly SOS tensors, and studied their properties.    We also introduced genuinely positive semi-definite tensors and barren tensors, and studied their properties.   Several outstanding questions remain.  We list them here.
	
	1. We now have two examples of barren tensors.   One is an odd order symmetric copositive tensor, another is an odd order symmetric B-tensor.  Otherwise we know very little about barren tensors.  Can we have more study on barren tensors?
	
	
	2. An even order symmetric B-tensor is always positive semi-definite.   However, we now have an odd order symmetric B-tensor is a barren tensor.   Is a B-tensor always a generalized positive semi-definite tensor?

	{3. Is the dual cone of SPSD$_{m, n}$   equal to SV$_{m, n}$ and whether SPSD$_{3, n}$ is a self-dual cone?}
	
	4. Give an example of a strongly positive semi-definite tensor, which is not a strongly SOS tensor.
		

	
	\bigskip	
	
	
	{{\bf Acknowledgment}}
	This work was partially supported by Research  Center for Intelligent Operations Research, The Hong Kong Polytechnic University (4-ZZT8),   the R\&D project of Pazhou Lab (Huangpu) (Grant no. 2023K0603),  the National Natural Science Foundation of China (Nos. 12471282 and 12131004), and the Fundamental Research Funds for the Central Universities (Grant No. YWF-22-T-204).
	
	

	{{\bf Data availability} Data will be made available on reasonable request.

		{\bf Conflict of interest} The authors declare no conflict of interest.}

	


\end{document}